\documentclass[12pt]{amsart} 
\usepackage{amsmath,enumitem}
\usepackage[english,  activeacute]{babel}
\usepackage{mathrsfs} 
\usepackage[latin1]{inputenc}
\usepackage{amssymb}
\usepackage{amsthm}
\usepackage{graphics}
\usepackage{array}
\usepackage{pdflscape}
\usepackage{a4wide}
\usepackage{color, url}
\usepackage{float}

\usepackage{graphicx}

\theoremstyle{plain}
\newtheorem{theorem}{Theorem}[section] 
\newtheorem{corollary}[theorem]{Corollary}
\newtheorem{lemma}[theorem]{Lemma}
\newtheorem{proposition}[theorem]{Proposition}
\newtheorem{conjecture}[theorem]{Conjecture}
\theoremstyle{definition}
\newtheorem{definition}[theorem]{Definition}
\newtheorem{example}[theorem]{Example}
\newtheorem{remark}[theorem]{Remark}

\def\st#1#2{\left[#1\atop#2\right]}
\def\sts#1#2{\left\{#1\atop  #2\right\}}

\newcommand{\Z}{{\mathbb Z}}

\def\li{\text{\rm Li}}

\def\piros#1{{\color{red}#1}}%
\def\kek#1{{\color{blue}#1}}%

\newcommand{\scB}{\mathscr{B}}

\newcommand{\calT}{\mathcal{T}}

\title[On the combinatorics of symmetrized poly-Bernoulli numbers]{On the combinatorics of symmetrized poly-Bernoulli numbers}

\author{Be\'ata B\'enyi}
\address{\noindent Faculty of Water Sciences, University of Public Service, Budapest, HUNGARY}
\email{beata.benyi@gmail.com}
\author{Toshiki Matsusaka}
\address{Institute for Advanced Research, Nagoya University, Furo-cho, Chikusa-ku, Nagoya 464-8601, JAPAN}
\email{toshikimatsusaka@gmail.com}

\date{\today}

\begin{document}
\begin{abstract}
In this paper we introduce three combinatorial models for symmetrized poly-Bernoulli numbers. Based on our models we derive generalizations of some identities for poly-Bernoulli numbers. Finally, we set open questions and directions of further studies.
\end{abstract}


\maketitle


\section{Introduction}

The symmetrized poly-Bernoulli numbers were introduced by Kaneko-Sakurai-Tsumura \cite{KST} in order to generalize the dual formula of poly-Bernoulli numbers. The poly-Bernoulli polynomials $B_n^{(k)}(x)$ of index $k\in \Z$ are defined by the generating function
\begin{align*}
\sum_{n=0}^{\infty}B_n^{(k)}(x)\frac{t^n}{n!}=e^{-xt}\frac{\li_k(1-e^{-t})}{1-e^{-t}},
\end{align*}
where $\li_k(z)$ is the polylogarithm function,
\begin{align*}
\li_k(z)=\sum_{m=1}^\infty \frac{z^m}{m^k}\quad (|z|<1).
\end{align*}
The two types of poly-Bernoulli numbers, $B_n^{(k)}$ and $C_{n}^{(k)}$ \cite{AIK,K97,K10} are special values of the poly-Bernoulli polynomials at $x=0$ and $x=1$.
\begin{align*}
B_n^{(k)}(0)=B_n^{(k)}\quad\mbox{and}\quad B_n^{(k)}(1)=C_n^{(k)}.
\end{align*}

For negative $k$ index these number sequences are integers (A099594 and A136126 \cite{OEIS}) and have several interesting combinatorial interpretations \cite{BH15, BH17, B08}.  


Both $B_n^{(-k)}$ and $C_n^{(-k)}$ are symmetric number arrays. These properties are special cases of the more general identity on poly-Bernoulli polynomials which hold for any non-negative integers $n$, $k$ and $m$. 
\begin{align*}
\sum_{j=0}^{m}\st{m}{j} B_{n}^{(-k-j)}(m)=\sum_{j=0}^{m}\st{m}{j} B_{k}^{(-n-j)}(m),
\end{align*}
where $\st{n}{k}$ is the (unsigned) Stirling number of the first kind which count the number of permutations of $[n]=\{1,2,\ldots, n\}$ with $k$ disjoint cycles.  

Kaneko-Sakurai-Tsumura \cite{KST} defined this expression as the \textit{symmetrized poly-Bernoulli numbers}.
\begin{align*}
\scB_n^{(-k)}(m):=\sum_{j=0}^{m}\st{m}{j} B_{n}^{(-k-j)}(m).
\end{align*}
Note that
\begin{align*}
\scB_n^{(-k)}(0)=B_n^{(-k)}\quad \mbox{and}\quad \scB_n^{(-k)}(1)=C_n^{(-k-1)}.
\end{align*}
The authors \cite{KST} suggested the combinatorial investigations of these number sequences. The first result in this direction is due to the second author. Matsusaka \cite{M20} showed that the alternating diagonal sums of symmetrized poly-Bernoulli numbers coincide with certain values of the Dumont-Foata polynomials/Gandhi polynomials.
\begin{align}\label{Gandhi}
\sum_{j=0}^{n} (-1)^j \scB_{n-j}^{(-j)}(k)=k!(-1)^{n/2}G_n(k),
\end{align}
where 
$G_n(z)$ denotes the Gandhi polynomials satisfying
\begin{align*}
G_{n+2}(z)=z(z+1)G_n(z+1)-z^2G_n(z)
\end{align*}
with $G_0(z)=1$ and $G_1(z)=0$.
Special cases of the theorem \cite{M20} are
\begin{align*}
\sum_{j=0}^{n} (-1)^j B_{n-j}^{(-j)} = \begin{cases}
	1, &\text{if } n = 0,\\
	0, &\text{if  } n > 0,
\end{cases}
\end{align*}
which was proven analytically in \cite{AK99} and combinatorially in \cite{BH15}, and
\begin{align*}
\sum_{j=0}^{n} (-1)^j C_{n-j}^{(-j-1)} = -G_{n+2},
\end{align*}
where $G_n := (2 - 2^{n+1}) B_n^{(1)} (1)$ are the Genocchi numbers $0,1, -1,0,1,0,-3,0,17,0,-155\ldots$ A001469 \cite{OEIS}. This last identity was proven by using analytical methods in \cite{KST}, but providing a combinatorial explanation is still open and seems to be a difficult problem. 

The paper is organized as follows. In the first three sections after the introduction we introduce three combinatorial models for the normalized symmetrized poly-Bernoulli numbers. In Section 5 we prove some recurrence relations. In the last section we formulate a conjecture and pose some open questions.




\section{Barred Callan sequences}


In this section we present a model of the \textit{normalized symmetrized poly-Bernoulli numbers} $\widehat\scB_n^k(m)$. We are interested in the combinatorics of symmetrized poly-Bernoulli numbers with negative $k$ indices (since these numbers are positive integers). Keeping the notation simpler, we define for non-negative integers $n$,$k$ and $m$,
\[
	\widehat{\scB}_n^k(m):=\frac{1}{m!}\scB_n^{(-k)}(m) \in \Z.
\]
In A099594 \cite{OEIS} Callan has given a combinatorial interpretation of the poly-Bernoulli numbers in certain type of permutations. Namely, $B_n^{(-k)}$ is the number of permutations of $[n+k]=\{1,\ldots, n+k \}$ such that all substrings of elements $\leq n$ and all substrings of elements $>n$ are in increasing order. Such permutations were called in the literature \cite{BH15,BH17} \emph{Callan permutations}. Essentially the same are \emph{Callan sequences} that we define as follows. Consider the set $N=\{\piros{1},\ldots, \piros{n}\}\cup\{\piros{*}\}$ (referred to as red elements) and  $K=\{\kek{1},\ldots, \kek{k}\}\cup\{\kek{*}\}$ (referred to as blue elements). Let $R_1,\ldots, R_r,R^*$ be a partition of the set $N$ into $r+1$ non-empty blocks ($0 \leq r \leq n$) and $B_1,\ldots,B_r,B^*$ a partition of the set of $K$ into $r+1$ non-empty blocks. The blocks containing $\kek{*}$ and $\piros{*}$ are denoted by $B^*$ and $R^*$, respectively. We call $B^*$ and $R^*$ \emph{extra blocks}, while the other blocks \emph{ordinary blocks}. We call a pair of a blue and  a red block, $(B_i;R_i)$  for an $i$ a \emph{Callan pair}. A Callan sequence is a linear arrangement of Callan pairs augmented by the extra pair 
\[(B_1;R_1)(B_2;R_2)\cdots(B_r;R_r)\cup(B^*;R^*).\]  

It is easy to check that this definition is equivalent with the one given by Callan in \cite{OEIS}. Given a Callan sequence, write the elements of the blocks in increasing order, record the blocks in the given order and if there are elements in $R^*$ besides $\piros{*}$ move this red elements into the front of the sequence, while the elements in $B^*$ at the end of the sequence. Delete $\kek{*}$ and $\piros{*}$, and shift the blue elements by $n$, $\kek{i}\rightarrow i+n$.

\begin{example}[All Callan sequences with $n = 2$ and $k=2$]
	\begin{align*}
		&(\kek{1},\kek{2},\kek{*};\piros{1},\piros{2};\piros{*}) & &(\kek{1},\kek{2};\piros{1},\piros{2})(\kek{*};\piros{*}), & &(\kek{1};\piros{1},\piros{2})(\kek{2},\kek{*};\piros{*}), & &(\kek{2};\piros{1},\piros{2})(\kek{1},\kek{*};\piros{*}), & &(\kek{1},\kek{2};\piros{1})(\kek{*};\piros{2},\piros{*}),\\
		&(\kek{1},\kek{2};\piros{2})(\kek{*};\piros{1},\piros{*}), & &(\kek{1};\piros{1})(\kek{2},\kek{*};\piros{2},\piros{*}), & &(\kek{2};\piros{1})(\kek{1},\kek{*};\piros{2},\piros{*}), & &(\kek{1};\piros{2})(\kek{2},\kek{*};\piros{1},\piros{*}), & &(\kek{2};\piros{2})(\kek{1},\kek{*};\piros{1},\piros{*}),\\
		&(\kek{1};\piros{1})(\kek{2};\piros{2})(\kek{*};\piros{*}), & &(\kek{1};\piros{2})(\kek{2};\piros{1})(\kek{*};\piros{*}), & &(\kek{2};\piros{1})(\kek{1};\piros{2})(\kek{*};\piros{*}), & &(\kek{2};\piros{2})(\kek{1};\piros{1})(\kek{*};\piros{*}).
	\end{align*}
	We list the corresponding Callan permutations in the same order as above
	\begin{align*}
		&\piros{12} \kek{34}, & &\kek{34}\piros{12}, & &\kek{3}\piros{12}\kek{4}, & &\kek{4}\piros{12}\kek{3}, & &\piros{2}\kek{34}\piros{1},\\
		&\piros{1}\kek{34}\piros{2}, & &\piros{2}\kek{3}\piros{1}\kek{4}, & &\piros{2}\kek{4}\piros{1}\kek{3}, & &\piros{1}\kek{3}\piros{2}\kek{4}, & &\piros{1}\kek{4}\piros{2}\kek{3},\\
		&\kek{3}\piros{1}\kek{4}\piros{2}, & &\kek{3}\piros{2}\kek{4}\piros{1}, & &\kek{4}\piros{1}\kek{3}\piros{2}, & &\kek{4}\piros{2}\kek{3}\piros{1}.
	\end{align*}
\end{example}

\begin{definition}
	For integers $n,k > 0$ and $m \geq 0$, the $m$-barred Callan sequence of size $n \times k$ is the Callan sequence with $m$ bars inserted between (before and after) the ordinary pairs. We let $\mathcal{C}_n^k(m)$ denote the number of all $m$-barred Callan sequences of size $n \times k$.
\end{definition}

\begin{example} [All $2$-barred Callan sequences with $n = 3$ and $k = 1$]
	\begin{align*}
		&||(\kek{1}, \kek{*}; \piros{1}, \piros{2}, \piros{3}, \piros{*}), & &||(\kek{1}; \piros{1}, \piros{2}, \piros{3}) (\kek{*}; \piros{*}), & &||(\kek{1}; \piros{1}, \piros{2})(\kek{*}; \piros{3}, \piros{*}), & &||(\kek{1}; \piros{1},\piros{3}) (\kek{*}; \piros{2}, \piros{*}),\\
		&||(\kek{1}; \piros{2}, \piros{3}) (\kek{*}; \piros{1}, \piros{*}), & &||(\kek{1}; \piros{1}) (\kek{*}; \piros{2}, \piros{3}, \piros{*}), & &||(\kek{1}; \piros{2}) (\kek{*}; \piros{1}, \piros{3}, \piros{*}), & &||(\kek{1}; \piros{3}) (\kek{*}; \piros{1}, \piros{2}, \piros{*}),\\
		&|(\kek{1}; \piros{1}, \piros{2}, \piros{3}) | (\kek{*}; \piros{*}), & &|(\kek{1}; \piros{1}, \piros{2}) | (\kek{*}; \piros{3}, \piros{*}), & &|(\kek{1}; \piros{1},\piros{3}) | (\kek{*}; \piros{2}, \piros{*}), & &|(\kek{1}; \piros{2}, \piros{3}) | (\kek{*}; \piros{1}, \piros{*}),\\
		&|(\kek{1}; \piros{1}) | (\kek{*}; \piros{2}, \piros{3}, \piros{*}), & &|(\kek{1}; \piros{2}) | (\kek{*}; \piros{1}, \piros{3}, \piros{*}), & &|(\kek{1}; \piros{3}) | (\kek{*}; \piros{1}, \piros{2}, \piros{*}),\\
	 	&(\kek{1}; \piros{1}, \piros{2}, \piros{3}) || (\kek{*}; \piros{*}), & &(\kek{1}; \piros{1}, \piros{2}) || (\kek{*}; \piros{3}, \piros{*}), & &(\kek{1}; \piros{1},\piros{3}) || (\kek{*}; \piros{2}, \piros{*}), & &(\kek{1}; \piros{2}, \piros{3}) || (\kek{*}; \piros{1}, \piros{*}),\\
		&(\kek{1}; \piros{1}) || (\kek{*}; \piros{2}, \piros{3}, \piros{*}), & &(\kek{1}; \piros{2}) || (\kek{*}; \piros{1}, \piros{3}, \piros{*}), & &(\kek{1}; \piros{3}) || (\kek{*}; \piros{1}, \piros{2}, \piros{*}).
	\end{align*}
\end{example} 
\begin{remark}
$m$-barred Callan sequences can be viewed in fact as a pair $(P, BP)$, where $P$ is a preferential arrangement of a subset of $\{1,2,\ldots,n\}$ and $BP$ is a barred preferential arrangement of a subset of $\{1,2,\ldots, k\}$. Barred preferential arrangements were introduced in \cite{AUP} and were used for combinatorial analysis of generalizations of geometric polynomials for instance in \cite{NBCC}.
\end{remark}
\begin{theorem}
The number $\mathcal{C}_n^k(m)$ of $m$-barred Callan sequences of size $n \times k$ is given by the normalized symmetrized poly-Bernoulli number $\widehat{\scB}_n^k(m)$.
\end{theorem}
\begin{proof}
Let $r$ be the number of ordinary pairs. Partition the elements of $N$ into $r+1$ blocks in $\sts{n+1}{r+1}$ ways, similarly $K$ into $r+1$ blocks in $\sts{k+1}{r+1}$ ways. ($\sts{n}{k}$ denotes the Stirling number of the second kind, counting the number of partitions of an $n$-element set into $k$ non-empty blocks.)
Order both types of ordinary blocks in $r!$ ways and choose the positions of the $m$ bars from the $r+1$ places between the ordinary blocks (note that repetition is allowed) in $\binom{r+1+m-1}{m}$ ways. By summing them up, we have
\begin{align}\label{Cal-exp}
\mathcal{C}_n^k(m)=\sum_{r=0}^{\min(n,k)}\binom{r+m}{m}(r!)^2\sts{n+1}{r+1}\sts{k+1}{r+1}.
\end{align}
By comparing this expression \eqref{Cal-exp} with the closed formula derived in \cite[(2.9)]{KST} for the symmetrized poly-Bernoulli numbers, the theorem follows.  
\end{proof}

It obviously follows from the definition that
\[
	\mathcal{C}_n^k(m) =\mathcal{C}_k^n(m).
\]

\begin{corollary}
A \emph{labeled} $m$-barred Callan sequence is an $m$-barred Callan sequence such that the bars are labeled. The number of labeled $m$-barred Callan sequences of size $n \times k$ is given by $\scB_n^{(-k)}(m)$. Clearly, $\scB_n^{(-k)}(m)=\scB_k^{(-n)}(m)$. 
\end{corollary}

By the right-hand side of (\ref{Cal-exp}), we define $\mathcal{C}_n^k(m)$ for $n= 0$ or $k=0$. Namely, $\mathcal{C}_n^0 (m) = \mathcal{C}_0^k (m) := 1$.

\begin{theorem}
	For integers $n \geq 0$ and $k > 0$, the number $\mathcal{C}_n^k(m)$ obeys the recurrence relation of
	\[
		\mathcal{C}_n^k(m) = \mathcal{C}_n^{k-1} (m) + \sum_{j=1}^n {n \choose j} \mathcal{C}_{n-j+1}^{k-1} (m) + m \sum_{j=1}^n {n \choose j} \mathcal{C}_{n-j}^{k-1} (m).
	\]
\end{theorem}

\begin{proof} 
	We count $m$-barred Callan sequences of size $n \times k$ according to the following cases. We let $|_\ell$ denote $\ell$ consecutive bars.
	\begin{itemize}
		\item[(0)] $|_m (\kek{1}, \kek{2}, \dots, \kek{k}, \kek{*}; \piros{1}, \piros{2}, \dots, \piros{n}, \piros{*})$.
		\item[$(1)$] $(\kek{1}, \kek{B}; \piros{R})$ is the first ordinary Callan pair with $\kek{B} \neq \emptyset$.
		\item[$(2)_\ell$] $|_\ell (\kek{1}; \piros{R})$ is the first ordinary Callan pair.
		\item[$(3)_\ell$] $(\kek{B'}; \piros{R})|_\ell (\kek{1}, \kek{B}; \piros{R'})$ for some $(\kek{B'}; \piros{R})$ and $\kek{B} \neq \emptyset$.
		\item[$(4)_0$] $(\kek{B'}; \piros{R}) |_0 (\kek{1}; \piros{R'})$ for some $(\kek{B'}; \piros{R})$.
		\item[$(4)_\ell$] $(\kek{B'}; \piros{R'}) |_\ell (\kek{1}; \piros{R})$ for some $\ell > 0$ and $(\kek{B'}; \piros{R'})$.
	\end{itemize}

	The cases (0) and $(1)$ are in bijection with $m$-barred Callan sequences of size $n \times (k-1)$ by deleting $\kek{1}$. So the number of such cases is $\mathcal{C}_n^{k-1} (m)$.
	
	Next, we consider the cases $(2)_0$, $(3)_\ell$, and $(4)_0$. In these cases, we delete $\kek{1}$ and $\piros{R}$, and insert the additional number $\piros{0}$ as follows. We assume that $\piros{R}$ contains $j$ elements. ($1 \leq j \leq n$).
	\begin{itemize}
		\item[$(2)_0$] Insert $\piros{0}$ into the extra red block.
			\[
				|_0 (\kek{1}; \piros{R}) |_{\ell'} (\kek{B'}; \piros{R'}) \cdots (\piros{R''}, \piros{*}; \kek{B''}, \kek{*}) \leftrightarrow |_{\ell'} (\kek{B'}; \piros{R'}) \cdots (\piros{0}, \piros{R''}, \piros{*}; \kek{B''}, \kek{*}).
			\]
			This gives $m$-barred Callan sequences of size $(n-j+1) \times (k-1)$ such that $\piros{0}$ is in the extra pair.
		\item[$(3)_\ell$] Replace $\piros{R}$ with $\piros{0}$.
			\[
				(\kek{B'}; \piros{R}) |_\ell (\kek{1}, \kek{B}; \piros{R'}) \leftrightarrow (\kek{B'}; \piros{0}) |_\ell (\kek{B}; \piros{R'}).
			\]
			This gives $m$-barred Callan sequences of size $(n-j+1) \times (k-1)$ such that $\piros{0}$ is alone in an ordinary pair.
		\item[$(4)_0$] Replace $\piros{R}$ with $\piros{0}$, and merge with $\piros{R'}$.
			\[
				(\kek{B'}; \piros{R}) |_0 (\kek{1}; \piros{R'}) \leftrightarrow (\kek{B'}; \piros{0}, \piros{R'}).
			\]
			This gives $m$-barred Callan sequences of size $(n-j+1) \times (k-1)$ such that the block that contains $\piros{0}$ includes also other red elements.
	\end{itemize}
	Clearly, the number of ways to create the $\piros{R}$ with $j$ elements is ${n \choose j}$. Thus, the number of patterns in the cases $(2)_0$, $(3)_\ell$ ($0 \leq \ell \leq m$), and $(4)_0$ is 
	\[
		\sum_{j=1}^n {n \choose j} \mathcal{C}_{n-j+1}^{k-1}(m).
	\]
	
	Finally, consider the remaining cases $(2)_\ell$ and $(4)_\ell$ with $1 \leq \ell \leq m$. If we delete the pair $(\kek{1}; \piros{R})$, we obtain $m$-barred Callan sequences of size $(n-j) \times (k-1)$. However, we obtain the same sequence $m$-times since $(\kek{1}; \piros{R})$ could have been after any bar. Indeed, conversely, take an $m$-barred Callan sequence of size $(n-j) \times (k-1)$ and insert the pair $(\kek{1}; \piros{R})$ after any bar. Thus, now we have 
	\[
		m \sum_{j=1}^n {n \choose j} \mathcal{C}_{n-j}^{k-1} (m).
	\]
	This concludes the proof.
\end{proof}

We give another type of recursion. Let $\widehat{\scB}_n^k(m;r)$ denote the number of $m$-barred Callan sequences with $r$ ordinary blocks. Then we have the following recursion. 
\begin{theorem} For positive integers $n,k >0$ and $m \geq 0$, it holds
\begin{align*}
\widehat{\scB}_n^{k}(m)=\sum_{j=1}^{n}\binom{n}{j}\sum_{r=0}^{\min{(n-j,k-1)}}(m+r+1)\widehat{\scB}_{n-j}^{k-1}(m;r)+
\sum_{r=0}^{\min{(n,k-1)}}(r+1)\widehat{\scB}_n^{k-1}(m; r).
\end{align*}
\end{theorem}
\begin{proof}
Consider an $m$-barred Callan sequence. There are two cases: $\kek{k}$ is in an ordinary pair as a singleton, or not, i.e., it is in an ordinary pair with other elements or in the extra pair. If it is in an ordinary pair as a singleton, let $j$ be the number of the red elements in this pair. Choose in $\binom{n}{j}$ ways such a Callan pair. Since it is an ordinary pair, $j$ is at least $1$. This new block can be inserted into the arrangement of the ordinary blocks and bars formed by the $m$-barred Callan sequence of size $(n-j) \times (k-1)$ with $r$ ordinary blocks, i.e., in $m+r+1$ ways. This gives the first part of our sum. 

On the other hand, if we insert $\kek{k}$ into any block that contains a blue element already, or into the extra block, that can be done in $r+1$ ways, which gives the second part of the sum.  
\end{proof}


\section{Weighted barred Callan sequences}


In this section we present a combinatorial interpretation, which allows us to extend the number that counted in our previous model the bars inserted between the Callan pairs, to arbitrary numbers. 

For this sake we introduce first a weight on permutations. Let $\pi$ be a permutation $\pi=\pi_1\pi_2\ldots\pi_n \in \mathfrak{S}_n$. Consider the maximal sequence $\pi_{i_0}>\pi_{i_1}>\pi_{i_2}>\cdots>\pi_{i_r}$, where $\pi_{i_0}=\pi_1$ and $\pi_{i_{j+1}}$ is the first element to the right of $\pi_{i_j}$ that is smaller for all $j$. Let $w(\pi)=r$, i.e.,  the length of this maximal sequence reduced by $1$. In other words, considering the elements of the permutation from left to right mark an element if it is smaller than the previous marked element. Then $w(\pi)$ is the number of marked elements reduced by one. For instance, for $\pi=\piros{8}\piros{6}9\piros{5}7\piros{2}34\piros{1}$ $w(\pi)=4$. Let $x^{\overline{n}}=x(x+1)(x+2)\cdots(x+n-1)$ denote the rising factorial. We have the following lemma.
\begin{lemma}
	\[
		\sum_{\pi \in \mathfrak{S}_n} x^{w(\pi)} = (x+1)^{\overline{n-1}}.
	\]
\end{lemma}
\begin{proof}
The left-hand side obeys the same recurrence as the right-hand side. The initial value is
$\sum_{\pi \in \mathfrak{S}_{1}} x^{w(\pi)} = 1$. By inserting the element $n$ into a permutation $\pi\in \mathfrak{S}_{n-1}$ the weight is increasing by one if we add it in the front as starting element, and stays preserved otherwise. Hence, 
\[
	\sum_{\pi\in \mathfrak{S}_{n}} x^{w(\pi)} = (x+n) \sum_{\pi \in \mathfrak{S}_{n-1}} x^{w(\pi)}.
\] 
\end{proof}
\begin{example}
\begin{align*}
		&\textcolor{red}{1}234, \textcolor{red}{1}243, \textcolor{red}{1}324, \textcolor{red}{1}342, \textcolor{red}{1}423, \textcolor{red}{1}432, \textcolor{red}{21}34, \textcolor{red}{21}43, \textcolor{red}{2}3\textcolor{red}{1}4, \textcolor{red}{2}34\textcolor{red}{1}, \textcolor{red}{2}4\textcolor{red}{1}3, \textcolor{red}{2}43\textcolor{red}{1},\\
		 &\textcolor{red}{31}24, \textcolor{red}{31}42, \textcolor{red}{321}4, \textcolor{red}{32}4\textcolor{red}{1}, \textcolor{red}{3}4\textcolor{red}{1}2, \textcolor{red}{3}4\textcolor{red}{21}, \textcolor{red}{41}23, \textcolor{red}{41}32, \textcolor{red}{421}3, \textcolor{red}{42}3\textcolor{red}{1}, \textcolor{red}{431}2, \textcolor{red}{4321}
	\end{align*}
We have
	\[
		\sum_{\pi \in \mathfrak{S}_4} x^{w(\pi)} = x^3 + 6x^2 + 11x + 6 = (x+1)(x+2)(x+3).
	\]
\end{example}   

We define now a weight on a $1$-barred Callan sequence (from now on barred Callan sequence) using the weight on permutations above. Let $\mathcal{B}_n^k$ denote the set of barred Callan sequences of size $n \times k$ and $\alpha\in \mathcal{B}_n^k$. The natural order of the blocks in a partition $\sigma=B_1/B_2/\ldots/B_n$ is given by the least elements. For instance, the blocks of the partition $\{1,3,9\}/\{2,4,7\}/\{5,6\}/\{8\}$ are listed in the natural order. We consider now the set of blue blocks of the Callan sequence with this natural order and add $|$ as the smallest element to the set. The weight $w(\alpha)$ is the weight of the permutation of the blue blocks (and the bar) in the barred Callan sequence $\alpha$.  

\begin{example} [All $1$-barred Callan sequences with $n = 2$ and $k = 2$ with indication of their weight]
	\begin{align*}
		&\underline{|}(\kek{1}, \kek{2}, \kek{*}; \piros{1}, \piros{2}, \piros{*}), & &\underline{|}(\kek{1}, \kek{2}; \piros{1}, \piros{2})(\kek{*};\piros{*}) & &\underline{|}(\kek{1}; \piros{1}, \piros{2})(\kek{2}, \kek{*}; \piros{*}), & &\underline{|}(\kek{2}; \piros{1}, \piros{2})(\kek{1}; \kek{*}; \piros{*}), & &\underline{|}(\kek{1}, \kek{2}; \piros{1}) (\kek{*}; \piros{2}, \piros{*}),\\
		&\underline{|}(\kek{1}; \piros{1})(\kek{2},\kek{*}; \piros{2}, \piros{*}), & &\underline{|}(\kek{2}; \piros{1})(\kek{1}, \kek{*}; \piros{2}, \piros{*}), & &\underline{|}(\kek{1}, \kek{2}; \piros{2}) (\kek{*}; \piros{1}, \piros{*}), & &\underline{|}(\kek{1}; \piros{2}) (\kek{2}, \kek{*}; \piros{1}, \piros{*}), & &\underline{|}(\kek{2}; \piros{2})(\kek{1}, \kek{*}; \piros{1}, \piros{*}),\\
		&\underline{|}(\kek{1}; \piros{1})(\kek{2}; \piros{2})(\kek{*}; \piros{*}), & &\underline{|}(\kek{2}; \piros{1})(\kek{1}; \piros{2})(\kek{*}; \piros{*}), & &\underline{|}(\kek{1};\piros{2})(\kek{2}; \piros{1})(\kek{*}; \piros{*}), & &\underline{|}(\kek{2};\piros{2})(\kek{1}; \piros{1}) (\kek{*}; \piros{*}), & &(\underline{\kek{1}}, \kek{2}; \piros{1}, \piros{2})\underline{|} (\kek{*};\piros{*}),\\
		&(\underline{\kek{1}}; \piros{1}, \piros{2}) \underline{|} (\kek{2}, \kek{*}; \piros{*}), & &(\underline{\kek{2}}; \piros{1}, \piros{2}) \underline{|} (\kek{1}; \kek{*}; \piros{*}), & &(\underline{\kek{1}}, \kek{2}; \piros{1}) \underline{|} (\kek{*}; \piros{2}, \piros{*}), & &(\underline{\kek{1}}; \piros{1}) \underline{|} (\kek{2},\kek{*}; \piros{2}, \piros{*}), & &(\underline{\kek{2}}; \piros{1}) \underline{|} (\kek{1}, \kek{*}; \piros{2}, \piros{*}),\\
		&(\underline{\kek{1}}, \kek{2}; \piros{2}) \underline{|} (\kek{*}; \piros{1}, \piros{*}), & &(\underline{\kek{1}}; \piros{2}) \underline{|} (\kek{2}, \kek{*}; \piros{1}, \piros{*}), & &(\underline{\kek{2}}; \piros{2}) \underline{|} (\kek{1}, \kek{*}; \piros{1}, \piros{*}), & &(\underline{\kek{1}}; \piros{1}) \underline{|} (\kek{2}; \piros{2})(\kek{*}; \piros{*}), & &(\underline{\kek{2}}; \piros{1}) \underline{|} (\kek{1}; \piros{2})(\kek{*}; \piros{*}),\\
		&(\underline{\kek{1}};\piros{2}) \underline{|} (\kek{2}; \piros{1})(\kek{*}; \piros{*}), & &(\underline{\kek{2}};\piros{2}) \underline{|} (\kek{1}; \piros{1}) (\kek{*}; \piros{*}), & &(\underline{\kek{1}}; \piros{1})  (\kek{2}; \piros{2}) \underline{|} (\kek{*}; \piros{*}), & &(\underline{\kek{2}}; \piros{1}) (\underline{\kek{1}}; \piros{2}) \underline{|} (\kek{*}; \piros{*}), & &(\underline{\kek{1}};\piros{2}) (\kek{2}; \piros{1}) \underline{|} (\kek{*}; \piros{*}),\\
		&(\underline{\kek{2}};\piros{2})(\underline{\kek{1}}; \piros{1}) \underline{|} (\kek{*}; \piros{*}).
	\end{align*}
\end{example}

\begin{definition}
We define the \emph{Callan polynomial} for any positive integers $n$ and $k$ as
\[C_n^k(x)=\sum_{\alpha\in \mathcal{B}_n^k}x^{w(\alpha)}.\]
\end{definition}

By the above example, we see that $C_2^2(x) = 2x^2 + 15x + 14$.

\begin{proposition}\label{C-exp}
The polynomials $C_n^k(x)$ are given by
\[C_n^k(x)=\sum_{j=0}^{\min(n,k)}j!(x+1)^{\overline{j}}\sts{n+1}{j+1}\sts{k+1}{j+1}.\]
\end{proposition}

\begin{proof}
It is straightforward from the definition of barred Callan sequences and the definition of the weight. 
\end{proof}

Next, we show the recursion by modifying the proof appropriately in the previous section.
We define $C_n^0(x)=C_0^k(x)=1$. 

\begin{theorem}\label{Callan-poly} For any integers $n \geq 0$ and $k > 0$, we have 
\begin{align}\label{recperm}
C_n^{k}(x)=C_{n}^{k-1}(x)+\sum_{j=1}^n\binom{n}{j}C_{n-j+1}^{k-1} (x) +x\sum_{j=1}^{n}\binom{n}{j}C_{n-j}^{k-1} (x).
\end{align}
\end{theorem}
\begin{proof}
We split the set $\mathcal{B}_n^k$  into disjoint subsets as follows:
 Let $A$ denote the set $\alpha\in \mathcal{B}_n^k$ such that $\kek{k}$ is in the extra pair with $\kek{*}$. Let $B$ denote the set $\beta\in \mathcal{B}_{n}^k$ such that $\kek{k}$ is in the first Callan pair alone and there is no bar before it. Let $C$ denote the set $\gamma\in \mathcal{B}_n^k$ such that $\kek{k}$ is in an ordinary block. Further, if it is alone in the first Callan pair, then the bar is before it.

If $\kek{k}$ is in the extra blue block $B^*$, we simply take a barred Callan sequence with $k-1$ blue elements and $n$ red elements and insert $\kek{k}$ into the extra block. The extra block does not affect the weight. Thus, we have
\[\sum_{\alpha\in A}x^{w(\alpha)}=C_{n}^{k-1}(x).\]

We obtain a Callan sequence $\beta\in B$ by choosing in $\binom{n}{j}$ ways $j$ red elements for the first Callan pair $(\kek{k}; \piros{R_1})$, and constructing from the remaining $n-j$ red elements and $k-1$ blue elements a barred Callan sequence. $(\kek{k}; \piros{R_1})$ is glued simply before the sequence. The weight will be increased by one, since the block $(\kek{k}; \piros{R_1})$ is the greatest among the blocks. Hence, we have
\[\sum_{\beta\in B}x^{w(\beta)}=x\sum_{j=1}^{n}\binom{n}{j}C_{n-j}^{k-1}(x).\]

We split the set $C$ into further disjoint subsets as follows. $C_1$ are the Callan sequences, where $\kek{k}$ is alone in its ordinary  block and the bar is directly before it. $C_2$ consists of the Callan sequences, where $\kek{k}$ is alone, a bar is not before it and it is not in the first Callan pair. Finally, $C_3$ are the Callan sequences, where $\kek{k}$ is not alone in its blue block. Clearly, $C=C_1\dot{\cup}C_2\dot{\cup}C_3$.

Choose again $j$ red elements in $\binom{n}{j}$ ways for $\kek{k}$ to create a block $(\kek{k};\widehat{\piros{R}})$. Construct a barred Callan sequence with $([n] \backslash \widehat{\piros{R}} )\cup \{\piros{0}\}$ red elements and $[k-1]$ blue elements. We have three cases:

If $\piros{0}$ is in the extra block, delete $\piros{0}$ and insert $(\kek{k};\widehat{\piros{R}})$ directly after the bar. In this case we obtain the set $C_1$. The weight does not change, since $\kek{k}$ is ``greater'' than $|$.  

If $\piros{0}$ is in an ordinary block and there is no other red element in its block, merge $(\kek{k};\widehat{\piros{R}})$ to this Callan pair by $(\kek{B}; \piros{0}) \to (\kek{B}, \kek{k}; \widehat{\piros{R}})$. Do not change the position of the bar. This case gives the set $C_3$. The weight does not change, since the so obtained blue block contains smaller elements than $\kek{k}$, and the order of the blocks are determined by their least elements.

If $\piros{0}$ is in an ordinary pair, say $(\kek{B}; \piros{0}, \piros{R})$ and this block contains other red elements also, then delete $\piros{0}$ and insert $(\kek{k};\widehat{\piros{R}})$ after this Callan pair, that is, $(\kek{B}; \piros{0}, \piros{R}) \to (\kek{B}; \piros{R}) (\kek{k}; \widehat{\piros{R}})$. If the bar was directly after this pair $(\kek{B}; \piros{0}, \piros{R})$, then delete it from here and place it now after $(\kek{k};\widehat{\piros{R}})$. This case gives the set $C_2$. The weight does not change since there is a block with smaller value (respecting to the order of blocks) to the left of the block $(\kek{k};\widehat{\piros{R}})$, hence, $(\kek{k};\widehat{\piros{R}})$ does not affect the weight anymore.

We have
\[\sum_{\gamma\in C}x^{w(\gamma)}=\sum_{j=1}^n\binom{n}{j}C_{n-j+1}^{k-1}(x),\]
which concludes the proof.
\end{proof}
\begin{corollary}
	For any integers $n, k \geq 0$ and $m \geq 0$, we have
	\[
		C_n^k(m) = \mathcal{C}_n^k(m) = \widehat{\scB}_n^k (m).
	\]
\end{corollary}


\section{Weighted alternative tableaux of rectangular shape} \label{s5}


In this section we introduce a weight on alternative tableaux of rectangular shapes and show that the so obtained polynomials are identical with the Callan polynomials, hence, the numbers of such tableaux are the normalized symmetrized poly-Bernoulli numbers. Alternative tableaux were introduced by Viennot \cite{Viennot}. The literature on alternative tableaux and related topics is extremely rich. For instance, a combinatorial interpretation of the generalized Dumont-Foata polynomial in terms of alternative tableaux was given in \cite{Josuat}.

\begin{definition} \cite[Definition 1.2]{N11}
	An \textit{alternative tableau} of rectangular shape of size $n \times k$ is a rectangle with a partial filling of the cells with left arrows $\leftarrow$ and down arrows $\downarrow$, such that all cells pointed by an arrow are empty. We let $\calT_n^k$ denote the set of all alternative tableaux of rectangular shape of size $n \times k$.
\end{definition}

\begin{example}\label{Ex-T}
	In Figure \ref{alttabl} we give an example of alternative tableaux of size $5 \times 6$ with its weight defined later.
	\begin{figure}[H]
		\centering
		\includegraphics[width=40mm]{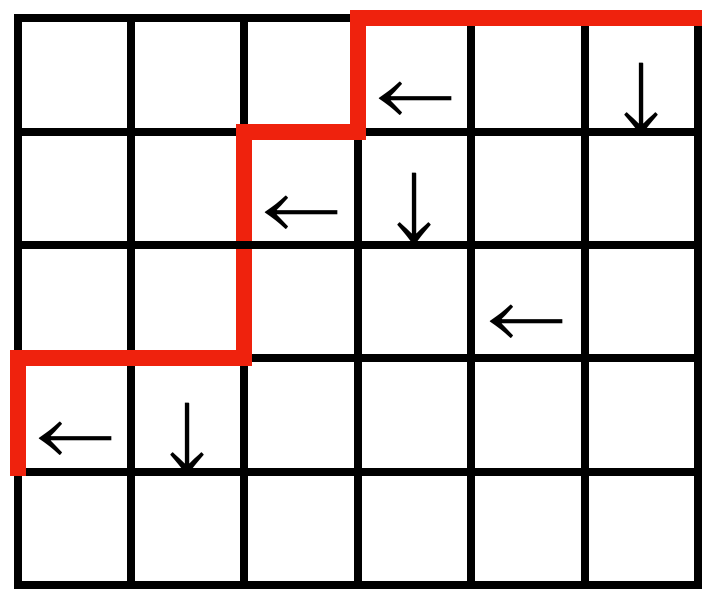} 
		\caption{An alternative tableaux of size $5 \times 6$}
		\label{alttabl}
	\end{figure}
\end{example}

 We introduce a weight on alternative tableaux as follows. For each $\lambda \in \calT_n^k$,
\begin{itemize}
	\item[1.] Consider the first (from the top) consecutive rows that contain left arrows $\leftarrow$.
	\item[2.] Count the number of left arrows $\leftarrow$ such that all $\leftarrow$ in the upper rows are located further to the right. 
\end{itemize}
We let $w(\lambda)$ denote the number of such left arrows. For instance, the alternative tableau in Figure \ref{alttabl} has the weight $w(\lambda) = 3$. In Figure \ref{altweight} we list all $31$ elements in $\calT_2^2$ with their weights.
\begin{figure}[H]
	\centering
	\includegraphics[width=160mm]{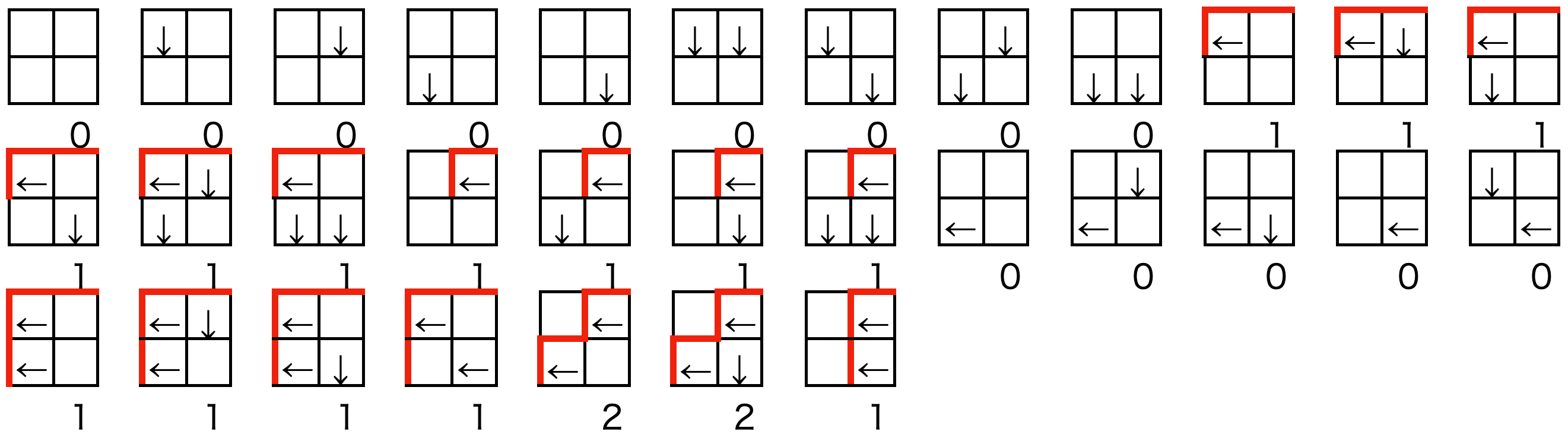}
	\caption{Alternative tableaux of size $2\times 2$ with their weights}\label{altweight}
\end{figure}

We define the polynomial $T_n^k(x)$ by
\[
	T_n^k(x) := \sum_{\lambda \in \calT_n^k} x^{w(\lambda)}.
\]
From the above example, $T_2^2(x) = 2x^2 + 15x + 14$, which coincides with the Callan polynomial $C_2^2(x)$. In general, the following holds.

\begin{theorem}
	We define $T_n^0(x) = T_0^k (x) = 1$. For any integers $n, k \geq 0$, the polynomial $T_n^k(x)$ coincides with the Callan polynomial $C_n^k(x)$.
\end{theorem}

\begin{proof}
	For each $\lambda \in \calT_n^k$, we let $R = R(\lambda)$ denote the most right column of $\lambda$. We split the set $\calT_n^k$ into disjoint subsets as follows: Let $A$ denote the set $\lambda \in \calT_n^k$ such that $R$ contains no $\leftarrow$. Let $B$ denote the set $\lambda \in \calT_n^k$ such that the top-right box is empty and $R$ contains at least one $\leftarrow$. Let $C$ denote the set $\lambda \in \calT_n^k$ such that the top-right box contains $\leftarrow$.
	
	If $\lambda \in A$, then $R$ is empty or contains the unique $\downarrow$. The remaining rectangle $\lambda^- := \lambda \backslash R$ defines a sub-rectangle in $\calT_n^{k-1}$, and we see that $w(\lambda) = w(\lambda^-)$. The number of patterns of $R$ is $n+1$ (empty or one $\downarrow$). Thus, we get
	\[
		\sum_{\lambda \in A} x^{w(\lambda)} = (n+1) T_n^{k-1}(x).
	\]
	
	If $\lambda \in B$, then $R$ contains $j \leftarrow$'s ($1 \leq j \leq n-1$). For each $j$, the number of patterns of $R$ is ${n \choose j+1}$, ($j \leftarrow$ and zero or one $\downarrow$). In the rectangle $\lambda \backslash R$, $j$ rows are killed, and the remaining rows define a sub-rectangle $\lambda^- \in \calT_{n-j}^{k-1}$. In this case it holds for the weight $w(\lambda^-) = w(\lambda)$, and hence
	\[
		\sum_{\lambda \in B} x^{w(\lambda)} = \sum_{j=1}^{n-1} {n \choose j+1} T_{n-j}^{k-1} (x) = \sum_{j=1}^{n-1} {n \choose j-1} T_j^{k-1} (x).
	\]
	
	Finally, if $\lambda \in C$, then $R$ contains $(j+1) \leftarrow$'s ($0 \leq j \leq n-1$). For each $j$, the number of patterns of $R$ is ${n \choose j+1}$. In the rectangle $\lambda \backslash R$, $(j+1)$ rows are killed, and the remaining rows define a sub-rectangle $\lambda^- \in \calT_{n-j-1}^{k-1}$. In this case, the $\leftarrow$ in the corner affect the weight of $\lambda$, thus $w(\lambda^-) = w(\lambda) - 1$. Hence, 
	\[
		\sum_{\lambda \in C} x^{w(\lambda)} = x \sum_{j=0}^{n-1} {n \choose j+1} T_{n-j-1}^{k-1} (x) = x \sum_{j=0}^{n-1} {n \choose j} T_j^{k-1} (x).
	\]
	
	Therefore, we have
	\begin{align}\label{T-rec}
		T_n^k(x) = (n+1) T_n^{k-1}(x) + \sum_{j=1}^{n-1} {n \choose j-1} T_j^{k-1} (x) + x \sum_{j=0}^{n-1} {n \choose j} T_j^{k-1} (x),
	\end{align}
	which is equivalent to the recursion formula for the Callan polynomial in Theorem \ref{Callan-poly}.
\end{proof}

\begin{corollary}
	For any integers $n, k \geq 0$ and $m \geq 0$, we have
	\[
		T_n^k(m) = \widehat{\scB}_n^k(m).
	\]
\end{corollary}


\section{Applications}


First, we present a generalization of Ohno-Sasaki's result on poly-Bernoulli numbers \cite[Theorem 1]{OS20} (see also \cite{OS20+}).
\begin{align}\label{OS-eq}
	\sum_{0 \leq i \leq \ell \leq m} (-1)^{i} \st{m+2}{i+1} B_{n+\ell}^{(-k)} = 0 \qquad (n \geq 0, m \geq k > 0),
\end{align}
The theorem gives a new type of recurrence relation for the (normalized) symmetrized poly-Bernoulli numbers $\widehat{\scB}_n^k(m)$ with the single index $k$, (see also a related question in \cite[Remark 14.5]{AIK}). 

\begin{theorem}\label{OS}
	For any $n \geq 0, m > k \geq 0$, we have
	\[
		\sum_{\ell = 0}^m (-1)^\ell \st{m+1}{\ell+1} C_{n+\ell}^k (x) = 0.
	\]
\end{theorem}

\begin{proof}
	By Proposition \ref{C-exp}, the left-hand side equals
	\begin{align*}
		\sum_{j=0}^k j! (x+1)^{\overline{j}} \sts{k+1}{j+1} \sum_{\ell = 0}^\infty (-1)^\ell \st{m+1}{\ell+1} \sts{n+\ell+1}{j+1}.
	\end{align*}
	By showing the identity
	\begin{align}\label{iden}
		\sum_{\ell = 0}^\infty (-1)^\ell \st{m+1}{\ell+1} \sts{n+\ell+1}{j+1} = 0 \qquad \text{for } j < m,
	\end{align}
	the theorem holds by the assumption $k < m$. We prove Equation \eqref{iden} by induction on $n$. Let $\delta_{i,j}$ denote the Kronecker delta defined by $\delta_{i,j} = 1$ if $i = j$ and $\delta_{i,j} = 0$ otherwise. For $n = 0$, by \cite[Proposition 2.6 (5.2)]{AIK}, we have
	\[
		\sum_{\ell=0}^\infty (-1)^\ell \st{m+1}{\ell+1} \sts{\ell+1}{j+1} = (-1)^j \delta_{j, m},
	\]
	which equals $0$ if $j <m$. For any positive $n$, by the recurrence relation of the Stirling numbers  of the second kind,
	\[
		\sum_{\ell=0}^\infty (-1)^\ell \st{m+1}{\ell+1} \sts{n+\ell+1}{j+1} = \sum_{\ell=0}^\infty (-1)^\ell \st{m+1}{\ell+1} \left(\sts{n+\ell}{j} + (j+1) \sts{n+\ell}{j+1} \right),
	\]
	which also equals to $0$ by the induction hypothesis.
\end{proof}

For example, since $C_n^k(0) = \widehat{\scB}_n^k(0) = B_n^{(-k)}$, we get
\begin{align}\label{Ber-rec}
	\sum_{\ell = 0}^m (-1)^\ell \st{m+1}{\ell+1} B_{n+\ell}^{(-k)} = 0 \qquad (n \geq 0, m > k \geq 0).
\end{align}
Our formula looks simpler than Ohno-Sasaki's formula (\ref{OS-eq}). Here we show the relation between these two results. Let $\mathrm{OS}(n)$ be the left-hand side of (\ref{OS-eq}). By a direct calculation, 
\begin{align*}
	\mathrm{OS}(n) - \mathrm{OS}(n+1) &=  \sum_{0 \leq i \leq \ell \leq m} (-1)^{i} \st{m+2}{i+1} B_{n+\ell}^{(-k)} + \sum_{1 \leq i \leq \ell \leq m+1} (-1)^i \st{m+2}{i} B_{n+\ell}^{(-k)}\\
		&= \sum_{\ell = 0}^m (-1)^\ell \st{m+2}{\ell+1} B_{n+\ell}^{(-k)} + \sum_{i = 1}^{m+1} (-1)^i \st{m+2}{i} B_{n +m+1}^{(-k)}.
\end{align*}
Since
\begin{align*}
	\sum_{j=0}^n \st{n+1}{j+1} x^j = (x+1)^{\overline{n}}
\end{align*}
and $\st{n}{n} = 1$ hold, the last sum becomes $\sum_{i=1}^{m+1} (-1)^i \st{m+2}{i} = (-1)^{m+1}$. Hence,
\[
	\mathrm{OS}(n) - \mathrm{OS}(n+1) = \sum_{\ell=0}^{m+1} (-1)^\ell \st{m+2}{\ell+1} B_{n+\ell}^{(-k)},
\]
which coincides with the left-hand side of (\ref{Ber-rec}) with shifted $m$ by one. This concludes that the equation (\ref{OS-eq}) implies (\ref{Ber-rec}). 

Next, we give another recurrence formula.
\begin{theorem}\label{diag-sum}
	For any integers $n, k \geq 0$, we have
	\[
		\sum_{\ell=0}^n \st{n+1}{\ell+1} C_\ell^k (x) = n! \sum_{j=0}^{\min(n,k)} (x+1)^{\overline{j}} \sts{k+1}{j+1} {n+1 \choose j+1}.
	\]
\end{theorem}

\begin{proof}
	By using Proposition \ref{C-exp} again, the left-hand side becomes
	\[
		\sum_{j=0}^\infty j! (x+1)^{\overline{j}} \sts{k+1}{j+1} \sum_{\ell=0}^n \st{n+1}{\ell+1}  \sts{\ell+1}{j+1}.
	\]
	The inner sum is an expression for the Lah numbers, which satisfies
	\[
		\sum_{\ell=0}^n \st{n+1}{\ell+1}  \sts{\ell+1}{j+1} = {n \choose j} \frac{(n+1)!}{(j+1)!}.
	\]
	Both sides of the identity counts the ways of partitions of $\{1,2,\ldots,n+1\}$ into $j+1$ linear arrangements, lists. In order to obtain a set of lists, split first the $n+1$ elements into $\ell+1$ cycles, then partition the $\ell+1$ cycles into $j+1$ blocks. The product of the cycles determines the list in a block.
On the other hand, take a permutation of $[n+1]$ and place bars to split it into $j+1$ pieces (from the $n$ places between the elements we choose $j$ to place the bars in $\binom{n}{j}$ ways). Since the order of the lists is irrelevant, we divide by the number of permutations of the lists, $(j+1)!$.

The theorem follows. 
\end{proof}

To apply the theorem for the special cases at $x = 0$ and $x=1$, we recall the following identity.

\begin{lemma}\label{Faul}
	For any integers $n > 0, k \geq 0$, we have
	\begin{align}\label{Seki}
		\sum_{j=0}^\infty j! \sts{k+1}{j+1} {n \choose j+1} = \sum_{i=1}^{n} i^k =:S_k(n).
	\end{align}
\end{lemma}

\begin{proof}
	Let $s_k(n)$ the left-hand side of (\ref{Seki}), and consider the generating function
	\[
		\sum_{k=0}^\infty s_k(n) \frac{t^k}{k!} = \sum_{j=0}^\infty j! {n \choose j+1} \sum_{k=0}^\infty \sts{k+1}{j+1} \frac{t^k}{k!} = e^t \sum_{j=0}^\infty {n \choose j+1} (e^t-1)^j.
	\]
	The last equality follows from the fact \cite[Proposition 2.6, (7)]{AIK}
	\[
		\sum_{k=0}^\infty \sts{k+1}{j+1} \frac{t^k}{k!} = \frac{e^t (e^t-1)^j}{j!}.
	\]
	This implies that
	\[
		\sum_{k=0}^\infty (s_k(n+1) - s_k(n)) \frac{t^k}{k!} = e^t \sum_{j=0}^n {n \choose j} (e^t-1)^j = e^{(n+1)t} = \sum_{k=0}^\infty (n+1)^k \frac{t^k}{k!},
	\]
	that is, $s_k(1) = 1$ and $s_k(n+1) = s_k(n) +(n+1)^k$. Hence $s_k(n) = S_k(n)$.
	
	We can also prove the equation
	\[
		(s_k(n+1) - s_k(n) = ) \sum_{j=0}^\infty j! \sts{k+1}{j+1} {n \choose j} = (n+1)^k
	\]
	combinatorially. The term $(n+1)^k$ counts the number of words $w_1 w_2 \cdots w_k$ of length $k$ out of an alphabet with $n+1$ distinct letters $\{0, 1, \dots, n\}$. We can get such a word as follows also: add the special position $w_0 := 0$ and partition the $k+1$ positions of the word into $j+1$ subsets, on the positions of a subset, the entries are the same. We choose the remaining entries in $j! {n \choose j}$ ways.
\end{proof}

\begin{corollary}
	At $x = 0$,
	\begin{align}\label{B-Seki}
		\sum_{\ell=0}^n \st{n+1}{\ell+1} B_\ell^{(-k)} = n! S_k(n+1),
	\end{align}
	At $x = 1$, we also get
	\begin{align}\label{diag-sum-C}
		\sum_{\ell=0}^n \st{n+1}{\ell+1} \widehat{\scB}_\ell^k (1) = n! (n+1)^{k+1}.
	\end{align}
\end{corollary}

\begin{proof}
The first equation (\ref{B-Seki}) immediately follows from Theorem \ref{diag-sum} and Lemma \ref{Faul}.

For the second equation (\ref{diag-sum-C}), by Theorem \ref{diag-sum} again, we have
\[
	\sum_{\ell=0}^n \st{n+1}{\ell+1} \widehat{\scB}_\ell^k (1) = n! \sum_{j=0}^{\min(n,k)} (j+1)! \sts{k+1}{j+1} {n+1 \choose j+1} = n! (n+1)^{k+1}.
\]
The last equation is given combinatorially by counting the term $(n+1)^{k+1}$ in a similar way as the proof of Lemma \ref{Faul}. In this argument, we do not need the special position $w_0$.
\end{proof}

We also give a direct combinatorial proof for the identity (\ref{diag-sum-C}). Both sides of the equation counts the number of permutations of $[n+(k+1)]$ such that tall substrings of consecutive elements greater than $n$ are in increasing order. Such a permutation can be decoded by a pair $(\pi, w)$, where $\pi \in \mathfrak{S}_n$, a permutation of $n$ and $w=w_1\ldots w_k w_{k+1}$ is a word of length $k+1$ on the alphabet $\{0,1,\ldots, n\}$. Let $\sigma$ be a permutation with the above property. Then the subsequence of the elements $\{1,2,\ldots, n\}$ is $\pi$, while $w_i$ is the number of the elements to the left of $i+n$ that are smaller than or equal to $n$. Clearly, the number of such pairs is given by $n!(n+1)^{k+1}$. For instance, for $n=7$ and $k=6$ the permutation  $\sigma={\bf 11}-6-{\bf 8-10}-3-1-{\bf 13-14}-7-5-4-2-{\bf 9-12}$ is decoded by the pair $(\pi;w)=(6-3-1-7-5-4-2;1710733)$.

On the other hand, we obtain such a permutation $\sigma$ using Callan sequences as follows. A \emph{$C$-Callan permutation} is a Callan permutation starting with an element greater than $n$. Equivalently, a \emph{$C$-Callan sequence} is a ($0$-barred) Callan sequence of size $n \times k$ with an extra red block $R^* = \{\piros{*}\}$. It can be shown that $C$-Callan permutations are in bijection with $1$-barred Callan sequences, and hence, their number is $B_n^{(-k)} (1) = \widehat{\scB}_n^{k-1}(1)$. Take a $C$-Callan sequence with red elements $\{\piros{1}, \dots, \piros{\ell}, \piros{*}\}$ and blue elements $\{\kek{1}, \dots, \kek{k}, \kek{k+1}, \kek{*}\}$. By the definition of the $C$-Callan sequence, this ends with $(\kek{B}\cup\{ \kek{*}\}; \piros{*})$. Construct a permutation of $\{0, 1, \dots, n\}$ with $\ell+1$ cycles $c_0, c_1, \dots, c_\ell$ in $\st{n+1}{\ell+1}$ ways. Let $c_i$ denote the $i$th cycle in the natural order of the cycles determined by the smallest elements of them. So for instance $c_0$ denotes the cycle that contains $0$. Replace in the $C$-Callan sequence $\kek{*} \piros{*}$ by $c_0$, and each red element for $i > 0$ by the cycle $c_i$ and take the product of the cycles in each red block. Finally, delete $0$ and shift the blue elements by $n$, $\kek{i} \to i+n$. The so obtained permutation is $\sigma$. For instance, the $C$-Callan sequence $(\kek{4}; \piros{4}) (\kek{1}, \kek{3}; \piros{1}) (\kek{6}, \kek{7}; \piros{2}, \piros{3}) (\kek{2}, \kek{5}, \kek{*}; \piros{*})$ and the cycles $c_0 = (0), c_1 = (1, 3), c_2 = (2, 7), c_3 = (4,5), c_4 = (6)$ with $n = 7, k = 6, \ell = 4$ correspond to the above $\sigma$ by
\begin{align*}
	(\kek{4}; \piros{4}) (\kek{1}, \kek{3}; \piros{1}) (\kek{6}, \kek{7}; \piros{2}, \piros{3}) &(\kek{2}, \kek{5}, \kek{*}; \piros{*}) \to \kek{4} (6) \kek{13} (1,3) \kek{67} (2,7)(4,5) \kek{25} (0)\\
		&\to \kek{4}-6-\kek{1}-\kek{3}-3-1-\kek{6}-\kek{7}-7-5-4-2-\kek{2}-\kek{5}-0\\
		&\to {\bf 11}-6-{\bf 8-10}-3-1-{\bf 13-14}-7-5-4-2-{\bf 9-12}.
\end{align*}


\section{Further problems}


In section \ref{s5}, we define the weight $w_{\leftarrow}(\lambda) := w(\lambda)$ on alternative tableaux by using left arrows $\leftarrow$. We let $w_\downarrow (\lambda)$ denote another weight on alternative tableaux corresponding to down arrows similarly. More precisely, for each $\lambda \in \calT_n^k$, the weight $w_\downarrow(\lambda)$ is defined as follows.
\begin{itemize}
	\item[1.] Consider the first (from the right) consecutive columns that contain down arrows $\downarrow$.
	\item[2.] Count the number of down arrows $\downarrow$ such that all $\downarrow$ in the right-hand columns are located in the upper rows.
\end{itemize}
Figure \ref{altweight2} shows the list of all elements in $\calT_2^2$ with the weight $w_\downarrow(\lambda)$.
\begin{figure}[H]
	\centering
	\includegraphics[width=160mm]{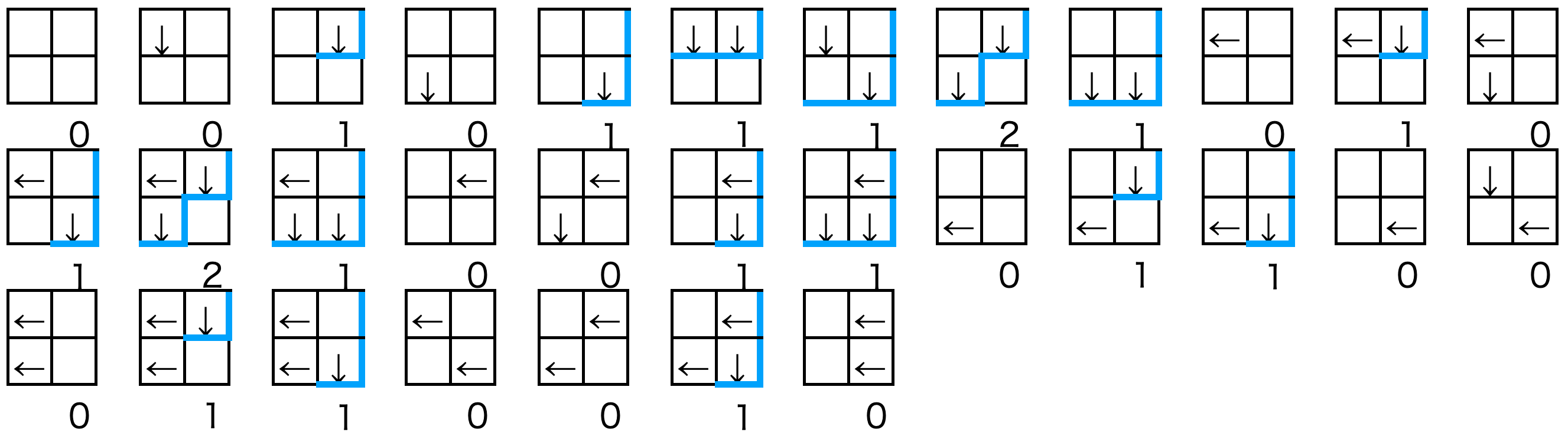}
	\caption{Alternative tableaux of size $2 \times 2$ with their weights $w_\downarrow(\lambda)$}\label{altweight2}
\end{figure}

We define the two-variable polynomial
\[
	T_n^k(x,y) := \sum_{\lambda \in \calT_n^k} x^{w_\leftarrow(\lambda)} y^{w_\downarrow(\lambda)}.
\]
From the above example, $T_2^2(x,y) = x^2 y + xy^2 + x^2 + 7xy + y^2 + 7x + 7y + 6$. By simple observations, we also see that
\[
	T_n^1(x,y) = T_1^n (x,y) = (2^{n-1} -1)xy + 2^{n-1} x + 2^{n-1} y + 2^{n-1}.
\]

\begin{conjecture}
	We put
	\begin{align*}
		t_n^0(x,y) &= t_0^k(x,y) = 1,\\
		t_n^1(x,y) &= t_1^n(x,y) = (2^{n-1}-1)xy + 2^{n-1}x + 2^{n-1}y + 2^{n-1}
	\end{align*}
	as initial values. The polynomials defined by
	\begin{align*}
		t_n^k(x,y) &:= \sum_{j=0}^n {n+1 \choose j} t_j^{k-1} (x,y) + (x-1) \sum_{j=0}^{n-1} {n \choose j} t_j^{k-1} (x,y)\\
			& \qquad + (y-1) \sum_{j=0}^{n-1} {n \choose j} t_j^{k-1} (x,y) + (x-1)(y-1) \sum_{j=0}^{n-1} {n-1 \choose j} t_j^{k-1} (x,y)
	\end{align*}
	coincide with $T_n^k(x,y)$.
\end{conjecture}
We checked the coincidence for $(n,k) = (2,2), (3,2)$, and $(2,3)$ by hand. By comparing the recurrence formula at $x =1$ or $y=1$ with that in (\ref{T-rec}), we easily see that $t_n^k(x,1) = t_n^k(1,x) = T_n^k(x)$. 

Another direction is to consider the polynomial at other values, for instance at negative integers. 

By applying Theorem \ref{OS} for $n=-1$ formally, we get
\[
	``C_k^{-1} (x) = -\frac{1}{m!} \sum_{\ell=1}^m (-1)^\ell \st{m+1}{\ell+1} C_k^{\ell-1} (x)".
\]
Here we used the symmetric property $C_n^k(x) = C_k^n(x)$. Recalling the condition $m > k \geq 0$ on $m$, and specializing by $m = k+1$, we tentatively define $C_k^{-1}(x)$ by
\[
	C_k^{-1}(x) := \frac{1}{(k+1)!} \sum_{\ell=0}^k (-1)^\ell \st{k+2}{\ell+2} C_k^\ell (x).
\]
\begin{proposition}
	For any integer $k \geq 0$, we have
	\[
		C_k^{-1} (x) = -\frac{S_k(-x)}{x},
	\]
	where $S_k(x)$ is the Seki-Bernoulli polynomial \cite[Section 1.2]{AIK} defined by
	\[
		S_k(x) := \frac{1}{k+1} \sum_{j=0}^k {k+1 \choose j} B_j x^{k+1-j}
	\]
	with the classical Bernoulli number $B_k = B_k^{(1)} (0)$.
\end{proposition}

\begin{proof}
	Let $s_k(x) := x C_k^{-1}(-x)$. By Proposition \ref{C-exp},
	\[
		s_k(x) = x C_k^{-1} (-x) = -\frac{1}{(k+1)!} \sum_{j=0}^k j! (-x)^{\overline{j+1}} \sts{k+1}{j+1} \sum_{\ell=0}^\infty (-1)^\ell \st{k+2}{\ell+2} \sts{\ell+1}{j+1}.
	\]
	Since the inner sum over $\ell$ equals $(-1)^j (k+1)!/(j+1)!$ for $0 \leq j \leq k$ (we prove it in Lemma \ref{last-lem}), we have
	\begin{align*}
		s_k(x) = \sum_{j=0}^k (-1)^{j+1} (-x)^{\overline{j+1}} \sts{k+1}{j+1} \frac{1}{j+1}.
	\end{align*}
	For any positive integer $n > 0$, 
	\[
		s_k(n) = \sum_{j=0}^k j! {n \choose j+1} \sts{k+1}{j+1}.
	\]
	By Lemma \ref{Faul}, this equals $S_k(n)$. Since $s_k(x)$ and $S_k(x)$ are polynomials, this concludes the proof, that is, $s_k(x) = S_k(x)$.
\end{proof}

\begin{lemma}\label{last-lem}
	For integers $k \geq j \geq 0$, we have
	\[
		\sum_{\ell=j}^k (-1)^{\ell+j} \st{k+2}{\ell+2} \sts{\ell+1}{j+1} = \frac{(k+1)!}{(j+1)!}.
	\]
\end{lemma}

\begin{proof}
	Consider the generating function of the left-hand side with respect to $k$. By \cite[Proposition 2.6 (7) and (9)]{AIK}, 
	\begin{align*}
		\sum_{k=j}^\infty \sum_{\ell=j}^k (-1)^{\ell+j} \st{k+2}{\ell+2} \sts{\ell+1}{j+1} \frac{t^{k+1}}{(k+1)!} &= \sum_{\ell=j}^\infty (-1)^{\ell+j}\sts{\ell+1}{j+1} \sum_{k=\ell}^\infty \st{k+2}{\ell+2} \frac{t^{k+1}}{(k+1)!}\\
		&= \frac{(-1)^{j+1}}{1-t} \sum_{\ell=j}^\infty \sts{\ell+1}{j+1} \frac{(\log(1-t))^{\ell+1}}{(\ell+1)!}\\
		&= \frac{1}{(j+1)!} \frac{t^{j+1}}{1-t} = \sum_{k=j}^\infty \frac{(k+1)!}{(j+1)!} \frac{t^{k+1}}{(k+1)!}
	\end{align*}
	This concludes the proof.
\end{proof}

One natural question is whether there exists a suitable generalization of the Callan polynomial $C_n^k(x)$ or the symmetrized poly-Bernoulli numbers $\widehat{\scB}_n^k(m)$ for negative integers $k$ and $m$.

It would be interesting to investigate the polynomials that arise by the weight function on alternative tableaux of other special shapes or on arbitrary shapes. 

In this paper we did not provide bijections between our models. It would be interesting to find simple bijections, especially between alternative tableaux and the Callan sequences. Also there should exist combinatorial proofs of Theorem \ref{OS} and so on.

\section*{Acknowledgements}

We would like to thank Yasuo Ohno and Yoshitaka Sasaki for sending us their preprint and some helpful comments. Further, we thank to Sithembele Nkonkobe for helpful conversations. The second author was supported by JSPS KAKENHI Grant Number 20K14292.

\end{document}